\documentclass[preprint,3p,11pt,authoryear]{elsarticle}

\makeatletter
\def\ps@pprintTitle{%
 \let\@oddhead\@empty
 \let\@evenhead\@empty
 \def\@oddfoot{\centerline{\thepage}}%
 \let\@evenfoot\@oddfoot}
\makeatother
\usepackage{amssymb,amsthm,mathtools}
\mathtoolsset{showonlyrefs=true} % Only numbers equations that are referenced

\usepackage{enumerate,natbib,caption,graphicx,subcaption,xcolor,geometry}
\usepackage{appendix} % to get appendix
\usepackage{mathrsfs} % to get \mathscr font
\usepackage{dsfont} % to get \mathds for indicator functions
\usepackage{cancel} % to cancel terms in equations

\usepackage[colorlinks]{hyperref}

\newtheorem{theorem}{Theorem}
\newtheorem{lemma}{Lemma}
\newtheorem{notation}{Notation}

\newcommand{\R}{\mathbb{R}}

\newcommand{\EE}{\mathbb{E}}

\newcommand{\VV}{\mathbb{V}\mathrm{ar}}

\newcommand{\OO}{\mathcal O}

\newcommand{\leqdef}{\vcentcolon=}

\newcommand{\rd}{{\rm d}}

\allowdisplaybreaks

\begin{document}

\begin{frontmatter}

    \title{Refined normal approximations for the Student distribution}%

    \author[a1,a2]{Fr\'ed\'eric Ouimet\texorpdfstring{}{)}}%

    \address[a1]{California Institute of Technology, Pasadena, CA 91125, USA.}%
    \address[a2]{McGill University, Montreal, QC H3A 0B9, Canada.}%

    \cortext[cor1]{Corresponding author}%
    \ead{frederic.ouimet2@mcgill.ca}%

    \begin{abstract}
        In this paper, we develop a local limit theorem for the Student distribution.
        We use it to improve the normal approximation of the Student survival function given in \cite{doi:10.1109/LCOMM.2015.2442576} and to derive asymptotic bounds for the corresponding maximal errors at four levels of approximation.
        As a corollary, approximations for the percentage points (or quantiles) of the Student distribution are obtained in terms of the percentage points of the standard normal distribution.
    \end{abstract}


\end{frontmatter}

\section{Introduction}\label{sec:intro}

    For any $\nu > 2$, the density function of the Student $t_{\nu}$ distribution is defined by
    \begin{equation}\label{eq:Student.density}
        f_{\nu}(x) = \frac{\Gamma\big(\frac{\nu + 1}{2}\big)}{\sqrt{\nu \pi} \, \Gamma\big(\frac{\nu}{2}\big)} \bigg(1 + \frac{x^2}{\nu}\bigg)^{-\frac{\nu + 1}{2}}, \quad x\in \R.
    \end{equation}
    For all $\nu > 2$, the mean and variance of $X\sim t_{\nu}$ are well known to be
    \begin{equation}\label{eq:noncentral.Student.mean.variance}
        \EE[X] = 0 \quad \text{and} \quad \VV(X) = \frac{\nu}{\nu - 2}.
    \end{equation}

    The first goal of our paper (Lemma~\ref{lem:LLT.Student}) is to establish a local asymptotic expansion for the ratio of the Student density \eqref{eq:Student.density} to the normal density with the same mean and variance, namely:
    \begin{equation}\label{eq:phi.M}
        \frac{1}{\sqrt{\nu/(\nu-2)}} \phi(\delta_x), \quad \text{where } \phi(z) \leqdef \frac{e^{-z^2/2}}{\sqrt{2\pi}} ~~ \text{and} ~~ \delta_x \leqdef \frac{x}{\sqrt{\nu/(\nu-2)}}.
    \end{equation}

    The second goal of the paper (Theorem~\ref{thm:refined.approximations}) is to prove a refined approximation of the survival function of the Student $t_{\nu}$ distribution and derive asymptotic bounds on the corresponding maximal errors.
    The most relevant publication in that direction is \cite{doi:10.1109/LCOMM.2015.2442576}, where the authors prove that, as $\nu\to \infty$,
    \begin{equation}\label{eq:most.relevant}
        \max_{a\in \R} \Big|\int_a^{\infty} f_{\nu - 1}(x) \rd x - \int_a^{\infty} \phi(x) \rd x\Big| \leq \frac{\widetilde{M}_0}{\nu} + \frac{\widetilde{C}_0}{\nu^2},
    \end{equation}
    for some universal constant $\widetilde{C}_0 > 0$, where
    \begin{equation}
        \widetilde{M}_0 = \frac{1}{4} \sqrt{\frac{7 + 5 \sqrt{2}}{\pi e^{1 + \sqrt{2}}}} = 0.1582\dots
    \end{equation}
    In Theorem~\ref{thm:refined.approximations}, we expand on this result by adding (asymptotic) correction terms to the lower end point of the Gaussian integral in \eqref{eq:most.relevant}. In total, we present four levels of approximation, up to an $\OO(\nu^{-4})$ precision.

    The third goal of the paper (Theorem~\ref{thm:percentage.points}) is to obtain approximations for the percentage points (or quantiles) of the Student distribution in terms of the percentage points of the standard normal distribution, the latter of which is usually more readily available. \cite{MR130734} makes a compendium of the known percentage point approximations for the noncentral Student distribution up to that point in time and compares them. Some of the approximations are based on the works of \cite{doi:10.2307/2983626,MR2072,doi:10.1214/aoms/1177732482,MR16611,MR0086449,doi:10.2307/2333195,doi:10.1093/biomet/44.1-2.219,doi:10.2172/4303062}.
    The best approximations at that time turns out to be related to those in \cite{doi:10.1214/aoms/1177732482}, \cite{doi:10.2307/2983626} and \cite{MR16611}.
    \begin{notation}
        Throughout the paper, the notation $u = \OO(v)$ means that $\limsup |u / v| < C$, as $\nu\to \infty$, where $C > 0$ is a universal constant.
        Whenever $C$ might depend on some parameters, we add a subscript (for example, $u = \OO_{\eta}(v)$).
    \end{notation}

\section{Normal approximations to the Student distribution}\label{sec:main.results}

    First, we need local approximations for the ratio of the Student density to the normal density function with the same mean and variance.

    \begin{lemma}[Local approximation]\label{lem:LLT.Student}
        For any $\nu > 2$ and $\eta\in (0,1)$, define
        \begin{equation}\label{eq:bulk}
            B_{\nu}(\eta) \leqdef \bigg\{x\in \R : \bigg|\frac{\delta_x}{\sqrt{\nu - 2}}\bigg| \leq \eta \, \nu^{-1/4}\bigg\},
        \end{equation}
        denote the bulk of the Student distribution.
        Then, as $\nu\to \infty$ and uniformly for $x\in B_{\nu}(\eta)$, we have
        \begin{equation}\label{eq:lem:LLT.Student.eq.log}
            \begin{aligned}
                \log\bigg(\frac{f_{\nu}(x)}{\frac{1}{\sqrt{\nu/(\nu-2)}} \phi(\delta_x)}\bigg)
                &= \nu^{-1} \bigg\{\frac{1}{4} \delta_x^4 - \frac{3}{2} \delta_x^2 + \frac{3}{4}\bigg\} \\[-2mm]
                &\quad+ \nu^{-2} \bigg\{-\frac{1}{6} \delta_x^6 + \frac{5}{4} \delta_x^4 - 3 \delta_x^2 + 1\bigg\} \\
                &\quad+ \nu^{-3} \bigg\{\frac{1}{8} \delta_x^8 - \frac{7}{6} \delta_x^6 + 4 \delta_x^4 - 6 \delta_x^2 + \frac{11}{8}\bigg\} + \OO_{\eta}\bigg(\frac{1 + |\delta_x|^{10}}{\nu^4}\bigg),
            \end{aligned}
        \end{equation}
        Furthermore,
        \begin{equation}\label{eq:lem:LLT.Student.eq}
            \begin{aligned}
                \frac{f_{\nu}(x)}{\frac{1}{\sqrt{\nu/(\nu-2)}} \phi(\delta_x)} = 1
                &+ \nu^{-1} \cdot \bigg\{\frac{1}{4} \delta_x^4 - \frac{3}{2} \delta_x^2 + \frac{3}{4}\bigg\} \\[-2mm]
                &+ \nu^{-2} \cdot \bigg\{\frac{1}{32} \delta_x^8 - \frac{13}{24} \delta_x^6 + \frac{41}{16} \delta_x^4 - \frac{33}{8} \delta_x^2 + \frac{41}{32}\bigg\} \\
                &+ \nu^{-3} \cdot \left\{\hspace{-1mm}
                    \begin{array}{l}
                        \frac{1}{384} \delta_x^{12} - \frac{17}{192} \delta_x^{10} + \frac{127}{128} \delta_x^8 - \frac{457}{96} \delta_x^6 \\[2mm]
                        + \frac{1357}{128} \delta_x^4 - \frac{651}{64} \delta_x^2 + \frac{281}{128}
                    \end{array}
                    \hspace{-1mm}\right\} + \OO_{\eta}\bigg(\frac{1 + |\delta_x|^{16}}{\nu^4}\bigg).
            \end{aligned}
        \end{equation}
    \end{lemma}

    For the interested reader, local approximations in the same vein as Lemma~\ref{lem:LLT.Student} were derived for the Poisson, binomial, negative binomial, multinomial, Dirichlet, Wishart and multivariate hypergeometric distributions in \cite[Lemma~2.1]{MR4213687}, \cite[Lemma~3.1]{MR4340237}, \cite[Lemma~2.1]{arXiv:2103.08846}, \cite[Theorem~2.1]{MR4249129}, \cite[Theorem~1]{MR4394974}, \cite[Theorem~1]{MR4358612}, \cite[Theorem~1]{MR4361955}, respectively.
    See also earlier references such as \cite{MR207011} (based on Fourier analysis results from \cite{MR14626}) for the Poisson, binomial and negative binomial distributions, and \cite{MR538319} for the binomial distribution.
    Another approach, using Stein's method, is used to study the variance-gamma distribution in \cite{MR3194737}. Also, Kolmogorov and Wasserstein distance bounds are derived in \cite{MR4291370,MR4064309} for the Laplace and variance-gamma distributions.

    By integrating the above local approximations, we can approximate the survival function of the Student $t_{\nu}$ distribution, i.e.,
    \begin{equation}
        S_{\nu}(a) \leqdef \int_a^{\infty} \hspace{-0.6mm} f_{\nu}(x) \rd x, \quad a\in \R,
    \end{equation}
    using the survival function of the normal distribution with the same mean and variance.

    \begin{theorem}[Survival function approximations]\label{thm:refined.approximations}
        As $\nu\to \infty$, we have
        \begin{align}
            &\hspace{-2cm}\text{Order 0 approximation:} \notag \\[-0.3mm]
            &E_0 \leqdef \max_{a\in \R} \Big|S_{\nu}(a) - \Psi(\delta_{a})\Big| \leq \frac{M_0}{\nu} + \frac{C_0}{\nu^2}, \label{eq:order.0.approx} \\[0.5mm]
            &\hspace{-2cm}\text{Order 1 approximation:} \notag \\[-0.3mm]
            &E_1 \leqdef \max_{a\in \R} \Big|S_{\nu}(a) - \Psi(\delta_{a - \frac{d_1}{\nu}})\Big| \leq \frac{M_1}{\nu^2} + \frac{C_1}{\nu^3}, \label{eq:order.1.approx} \\[0.5mm]
            &\hspace{-2cm}\text{Order 2 approximation:} \notag \\[-0.3mm]
            &E_2 \leqdef \max_{a\in \R} \Big|S_{\nu}(a) - \Psi(\delta_{a - (\frac{d_1}{\nu} + \frac{d_2}{\nu^2})})\Big| \leq \frac{M_2}{\nu^3} + \frac{C_2}{\nu^4}, \label{eq:order.2.approx} \\[0.5mm]
            &\hspace{-2cm}\text{Order 3 approximation:} \notag \\[-0.3mm]
            &E_3 \leqdef \max_{a\in \R} \Big|S_{\nu}(a) - \Psi(\delta_{a - (\frac{d_1}{\nu} + \frac{d_2}{\nu^2} + \frac{d_3}{\nu^3})})\Big| \leq \frac{C_3}{\nu^4}, \label{eq:order.3.approx}
        \end{align}
        where $\Psi$ denotes the survival function of the standard normal distribution, $C_i, ~i\in \{0,1,2,3\},$ are universal constants, and
        \begin{equation}\label{eq:d.k.M.k}
            \begin{aligned}
                d_1 &\leqdef \frac{\delta_a}{4} (\delta_a^2 - 3), \\
                d_2 &\leqdef - \frac{\delta_a}{96} (13 \delta_a^4 - 88 \delta_a^2 + 195), \\
                d_3 &\leqdef - \frac{\delta_a}{384} (35 \delta_a^6 - 293 \delta_a^4 + 1025 \delta_a^2 - 1767), \\[1mm]
                M_0 &\leqdef \max_{y\in \R} \frac{|y|}{4} |y^2 - 3| \phi(y) = 0.137647\dots, \\
                M_1 &\leqdef \max_{y\in \R} \frac{|y|}{96} |13 y^4 - 88 y^2 + 195| \phi(y) = 0.353017\dots, \\
                M_2 &\leqdef \max_{y\in \R} \frac{|y|}{384} |35 y^6 - 293 y^4 + 1025 y^2 - 1767| \phi(y) = 0.758112\dots.
            \end{aligned}
        \end{equation}
        The constants $M_1,M_2,M_3$ are illustrated in Figure~\ref{fig:asymptotics.E0.E1.E2} along with the corresponding rates of convergence.
    \end{theorem}

    As a corollary to Theorem~\ref{thm:refined.approximations}, we obtain asymptotic expansions for the percentage points (or quantiles) of the Student distribution in terms of the percentage points of the standard normal distribution.

    \begin{theorem}[Percentage point approximations]\label{thm:percentage.points}
        Let $\nu > 2$, and let $\alpha\in (0,1)$ be such that $\alpha = S_{\nu}(\lambda)$ for some $\lambda\in B_{\nu}(\eta)$ and $\eta\in (0,1)$.
        As $\nu\to \infty$, we have
        \begin{equation}
            \begin{aligned}
                \frac{\alpha}{\phi(\lambda)}
                &= \frac{\Psi(\lambda)}{\phi(\lambda)} + \left(\frac{\lambda}{\nu} + \frac{d_1}{\nu}\right) + \OO(\nu^{-2}), \\[1mm]
                \frac{\alpha}{\phi(\lambda)}
                &= \frac{\Psi(\lambda)}{\phi(\lambda)} + \left(\frac{\lambda}{\nu} + \frac{\lambda}{2 \nu^2} + \frac{d_1}{\nu} + \frac{d_2}{\nu^2} - \frac{d_1}{\nu^2}\right) \\
                &\quad+ \frac{\lambda}{2} \left(\frac{\lambda^2}{\nu^2} + \frac{2 \lambda d_1}{\nu^2} + \frac{d_1^2}{\nu^2}\right) + \OO(\nu^{-3}), \\[1mm]
                \frac{\alpha}{\phi(\lambda)}
                &= \frac{\Psi(\lambda)}{\phi(\lambda)} + \left(\frac{\lambda}{\nu} + \frac{\lambda}{2 \nu^2} + \frac{\lambda}{2 \nu^3} + \frac{d_1}{\nu} + \frac{d_2}{\nu^2} + \frac{d_3}{\nu^3} - \frac{d_1}{\nu^2} - \frac{d_2}{\nu^3} - \frac{d_1}{2 \nu^3}\right) \\
                &\quad+ \frac{\lambda}{2} \left(\frac{\lambda^2}{\nu^2} + \frac{2 \lambda d_1}{\nu^2} + \frac{d_1^2}{\nu^2} + \frac{\lambda^2}{\nu^3} + \frac{2 \lambda d_2}{\nu^3} - \frac{\lambda d^1}{\nu^3} + \frac{2 d_1 d_2}{\nu^3} - \frac{2 d_1^2}{\nu^3}\right) \\
                &\quad+ \frac{(\lambda^2 - 1)}{6} \left(\frac{\lambda^3}{\nu^3} + \frac{d_1^3}{\nu^3}\right) + \OO(\nu^{-4}).
            \end{aligned}
        \end{equation}
        We approximate the $100 \cdot (1 - \alpha) \%$ percentile of the Student distribution by solving numerically for $\lambda$ in one of the three equations above (ignoring the $\OO(\cdot)$ terms).
    \end{theorem}

    \begin{center}
    \begin{figure}[!ht]
        \captionsetup[subfigure]{labelformat=empty}
        \captionsetup{width=0.80\linewidth}
        \centering
        \begin{subfigure}[b]{0.46\textwidth}
            \centering
            \includegraphics[width=\textwidth, height=0.85\textwidth]{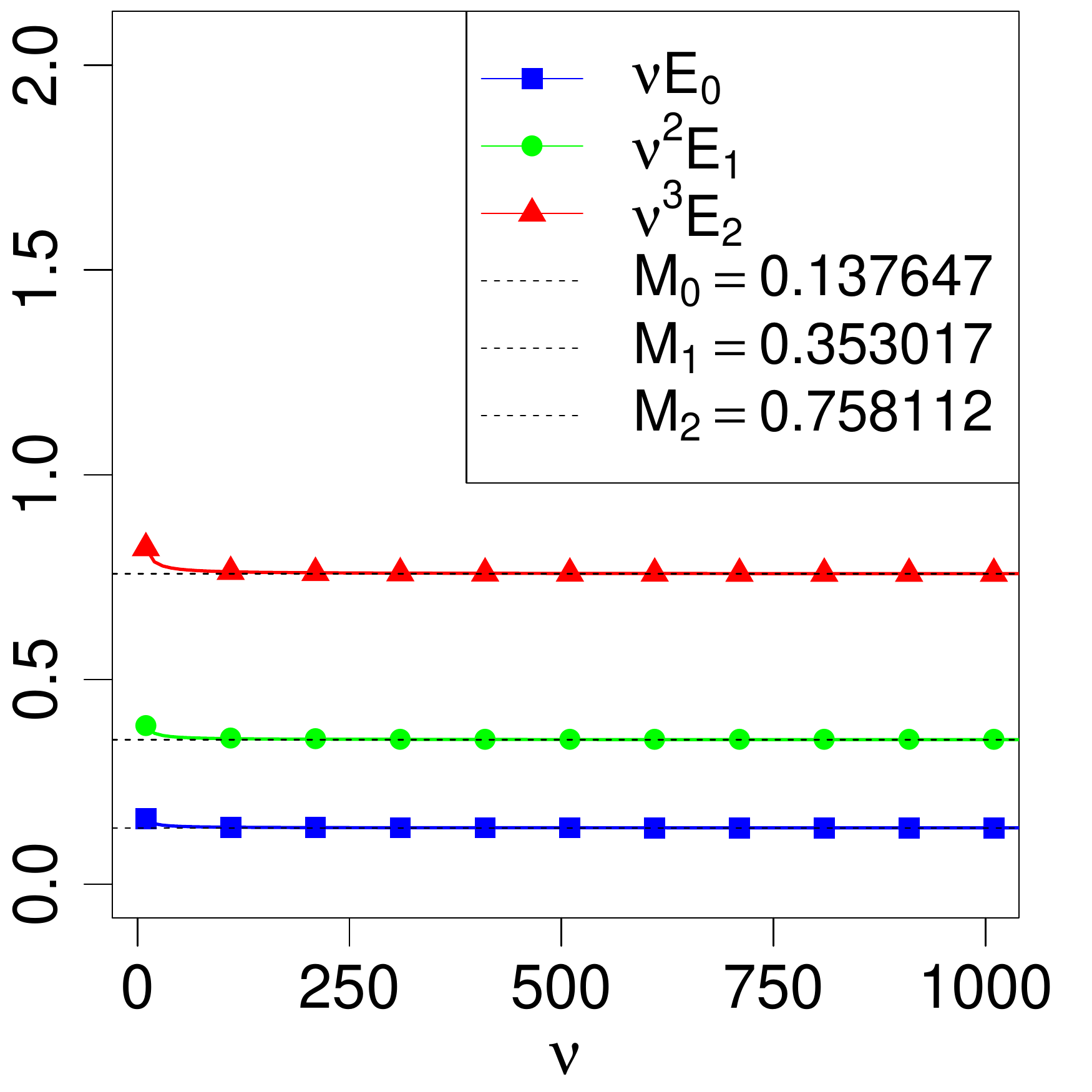}
        \end{subfigure}
        \quad
        \begin{subfigure}[b]{0.46\textwidth}
            \centering
            \includegraphics[width=\textwidth, height=0.85\textwidth]{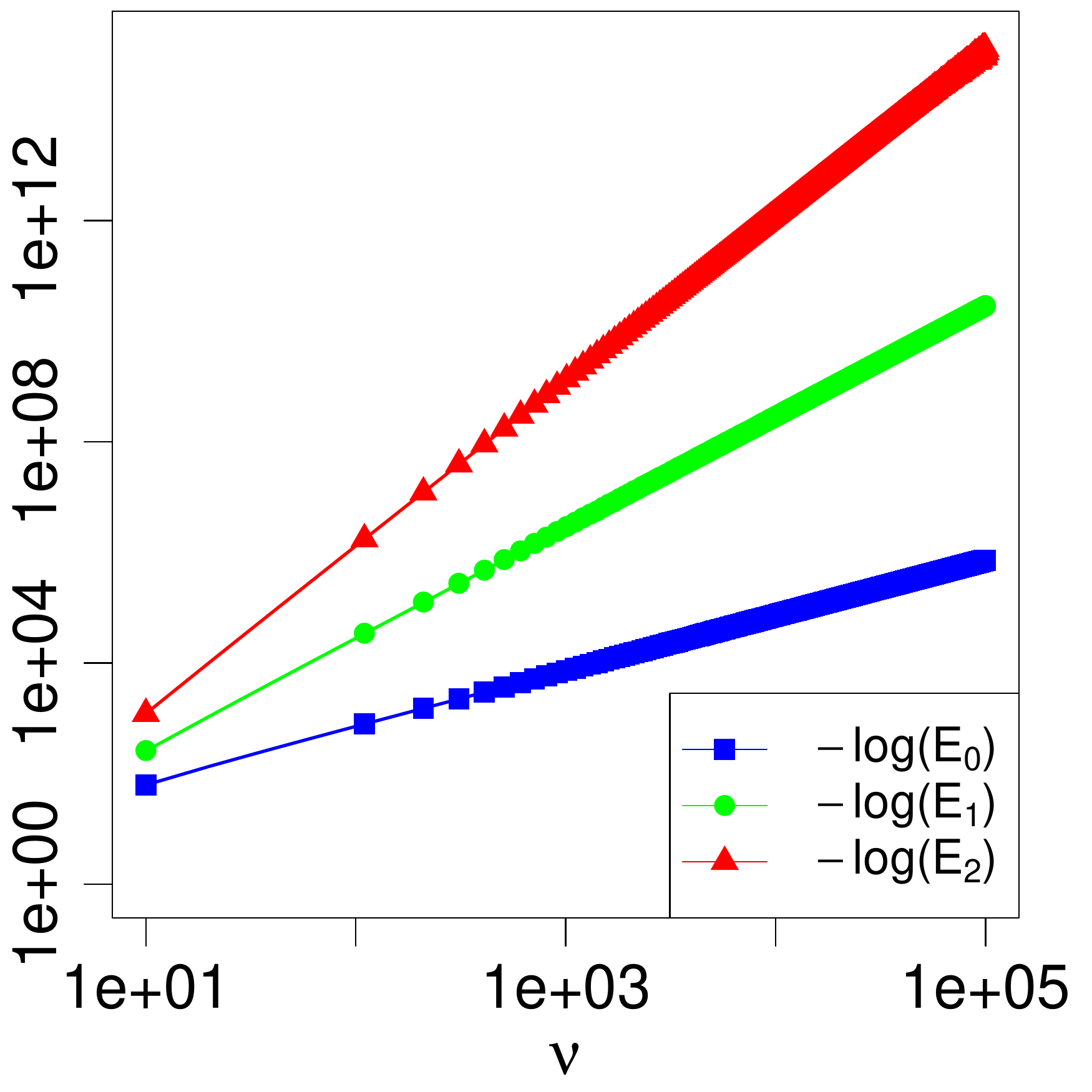}
        \end{subfigure}
        \caption{Numerical illustration of the asymptotic constants $M_i$ (on the left) and the log-log plot for the maximum absolute errors as a function of $\nu$ (on the right) for the first three levels of approximation.}
        \label{fig:asymptotics.E0.E1.E2}
        \vspace{-3mm}
    \end{figure}
    \end{center}

\section{Proofs}\label{sec:proofs}

    \begin{proof}[Proof of Lemma~\ref{lem:LLT.Student}]
        By taking the logarithm in \eqref{eq:Student.density}, we have
        \begin{equation}\label{eq:LLT.beginning}
            \begin{aligned}
                \log\bigg(\frac{f_{\nu}(x)}{\frac{1}{\sqrt{\nu/(\nu-2)}} \phi(\delta_x)}\bigg)
                &= - \frac{1}{2} \log \Big(\frac{\nu-2}{2}\Big) - \frac{\nu + 1}{2} \log \bigg(1 + \Big(\frac{\delta_x}{\sqrt{\nu - 2}}\Big)^2\bigg) \\[-1mm]
                &\quad+ \log \Gamma\Big(\frac{\nu + 1}{2}\Big) - \log \Gamma \Big(\frac{\nu}{2}\Big) + \frac{1}{2} \delta_x^2.
            \end{aligned}
        \end{equation}
        Using the expansions
        \begin{equation}
            \begin{aligned}
                \log \Gamma\Big(\frac{\nu + 1}{2}\Big)
                &= \Big(\frac{\nu}{2}\Big) \log \Big(\frac{\nu + 1}{2}\Big) - \frac{\nu}{2} - \frac{1}{2} + \frac{1}{2} \log (2\pi) \\
                &\quad+ \frac{2}{12 (\nu + 1)} - \frac{2^3}{360 (\nu + 1)^3} + \OO(\nu^{-4}), \\[1.5mm]
                \log \Gamma\Big(\frac{\nu}{2}\Big)
                &= \Big(\frac{\nu}{2} - \frac{1}{2}\Big) \log \Big(\frac{\nu}{2}\Big) - \frac{\nu}{2} + \frac{1}{2} \log (2\pi) \\
                &\quad+ \frac{2}{12 \nu} - \frac{2^3}{360 \nu^3} + \OO(\nu^{-5}),
            \end{aligned}
        \end{equation}
        (see, e.g., \cite[p.257]{MR0167642})
        and
        \begin{equation}
            \begin{aligned}
                \log \Gamma\Big(\frac{\nu + 1}{2}\Big) - \log \Gamma \Big(\frac{\nu}{2}\Big)
                &= \Big(\frac{\nu}{2}\Big) \log \Big(\frac{\nu + 1}{2}\Big) - \Big(\frac{\nu}{2} - \frac{1}{2}\Big) \log \Big(\frac{\nu}{2}\Big) \\
                &\quad- \frac{1}{2} + \frac{2}{12 (\nu + 1)} - \frac{2}{12 \nu} + \OO(\nu^{-4}) \\
                &= \frac{1}{2} \log \Big(\frac{\nu}{2}\Big) - \frac{1}{4 \nu} + \frac{1}{24 \nu^3} + \OO(\nu^{-4}),
            \end{aligned}
        \end{equation}
        we can rewrite \eqref{eq:LLT.beginning} as
        \begin{equation}\label{eq:LLT.beginning.next.1}
            \begin{aligned}
                \log\bigg(\frac{f_{\nu}(x)}{\frac{1}{\sqrt{\nu/(\nu-2)}} \phi(\delta_x)}\bigg)
                &= - \frac{1}{2} \log \Big(\frac{\nu-2}{2}\Big) - \frac{\nu + 1}{2} \log \bigg(1 + \Big(\frac{\delta_x}{\sqrt{\nu - 2}}\Big)^2\bigg) \\[-0.5mm]
                &\quad+\frac{1}{2} \log \Big(\frac{\nu}{2}\Big) - \frac{1}{4 \nu} + \frac{1}{24 \nu^3} + \frac{1}{2} \delta_x^2 + \OO(\nu^{-4}).
            \end{aligned}
        \end{equation}
        Using the Taylor expansions
        \begin{equation}
            - \frac{1}{2} \log \Big(\frac{\nu-2}{2}\Big) + \frac{1}{2} \log \Big(\frac{\nu}{2}\Big) = \frac{1}{\nu} + \frac{1}{\nu^2} + \frac{4}{3 \nu^3} + \OO(\nu^{-4}),
        \end{equation}
        and
        \begin{equation}\label{eq:Taylor.log.1.plus.y}
            \log(1 + y) = y - \frac{y^2}{2} + \frac{y^3}{3} - \frac{y^4}{4} + \OO_{\eta}(y^5), \quad |y| \leq \eta,
        \end{equation}
        we have
        \begin{equation}\label{eq:LLT.beginning.next.2}
            \begin{aligned}
                \log\bigg(\frac{f_{\nu}(x)}{\frac{1}{\sqrt{\nu/(\nu-2)}} \phi(\delta_x)}\bigg)
                &= \frac{3}{4 \nu} + \frac{1}{\nu^2} + \frac{11}{8 \nu^3} + \left[\frac{1}{2} - \frac{\nu + 1}{2 (\nu - 2)}\right] \delta_x^2 \\[-2mm]
                &\quad+ \frac{\nu + 1}{4 (\nu - 2)^2} \delta_x^4 - \frac{\nu + 1}{6 (\nu - 2)^3} \delta_x^6 + \frac{\nu + 1}{8 (\nu - 2)^4} \delta_x^8 \\[0.5mm]
                &\quad+ \OO_{\eta}\bigg(\frac{1 + |\delta_x|^{10}}{\nu^4}\bigg).
            \end{aligned}
        \end{equation}
        Now,
        \begin{equation}
            \begin{aligned}
                \frac{1}{2} - \frac{\nu + 1}{2 (\nu - 2)} &= -\frac{3}{2\nu} - \frac{3}{\nu^2} - \frac{6}{\nu^3} + \OO(\nu^{-4}), \\[0.5mm]
                \frac{\nu + 1}{4 (\nu - 2)^2} &= \frac{1}{4 \nu} + \frac{5}{4 \nu^2} + \frac{4}{\nu^3} + \OO(\nu^{-4}), \\[0.5mm]
                -\frac{\nu + 1}{6 (\nu - 2)^3} &= -\frac{1}{6 \nu^2} - \frac{7}{6 \nu^3} + \OO(\nu^{-4}), \\[0.5mm]
                \frac{\nu + 1}{8 (\nu - 2)^4} &= \frac{1}{8 \nu^3} + \OO(\nu^{-4}),
            \end{aligned}
        \end{equation}
        so we can rewrite \eqref{eq:LLT.beginning.next.2} as
        \begin{equation}\label{eq:LLT.beginning.next.3}
            \begin{aligned}
                \log\bigg(\frac{f_{\nu}(x)}{\frac{1}{\sqrt{\nu/(\nu-2)}} \phi(\delta_x)}\bigg)
                &= \frac{\frac{1}{4} \delta_x^4 - \frac{3}{2} \delta_x^2 + \frac{3}{4}}{\nu} + \frac{-\frac{1}{6} \delta_x^6 + \frac{5}{4} \delta_x^4 - 3 \delta_x^2 + 1}{\nu^2} \\[-2.5mm]
                &\quad+ \frac{\frac{1}{8} \delta_x^8 - \frac{7}{6} \delta_x^6 + 4 \delta_x^4 - 6 \delta_x^2 + \frac{11}{8}}{\nu^3} + \OO_{\eta}\bigg(\frac{1 + |\delta_x|^{10}}{\nu^4}\bigg),
            \end{aligned}
        \end{equation}
        which proves \eqref{eq:lem:LLT.Student.eq.log}.
        To obtain \eqref{eq:lem:LLT.Student.eq} and conclude the proof, we take the exponential on both sides of the last equation and we expand the right-hand side with
        \begin{equation}\label{eq:Taylor.exponential}
            e^y = 1 + y + \frac{y^2}{2} + \frac{y^3}{6} + \OO(e^{\widetilde{\eta}} y^4), \quad \text{for } -\infty < y \leq \widetilde{\eta}.
        \end{equation}
        For $\nu$ large enough and uniformly for $x\in B_{\nu}(\eta)$, the right-hand side of \eqref{eq:LLT.beginning.next.3} is $\OO(1)$.
        When this bound is taken as $y$ in \eqref{eq:Taylor.exponential}, it explains the error in \eqref{eq:lem:LLT.Student.eq}.
    \end{proof}

    \begin{proof}[Proof of Theorem~\ref{thm:refined.approximations}]
        By large deviation bounds, the approximations are trivial when $a\not\in B_{\nu}(1/2)$.
        Therefore, for the remainder of the proof, we assume that $a\in B_{\nu}(1/2)$.
        Let
        \begin{equation}
            c = \frac{d_1}{\nu} + \frac{d_2}{\nu^2} + \frac{d_3}{\nu^3},
        \end{equation}
        where $d_1,d_2,d_3\in \R$ are to be chosen later, then we have the Taylor expansion
        \begin{equation}\label{eq:c}
            \begin{aligned}
                \int_{\delta_{a-c}}^{\delta_a} \phi(y) \rd y
                &= \phi(\delta_a) \int_{\delta_{a-c}}^{\delta_a} \rd y + \phi'(\delta_a) \int_{\delta_{a-c}}^{\delta_a} (y - \delta_a) \rd y + \frac{\phi''(\delta_a)}{2} \int_{\delta_{a-c}}^{\delta_a} (y - \delta_a)^2 \rd y \\
                &\quad+ \OO\left(\frac{\phi'''(\delta_a)}{6} \int_{\delta_{a-c}}^{\delta_a} (y - \delta_a)^3 \rd y\right) \\
                &= \phi(\delta_a) \, \left\{\hspace{-1mm}
                \begin{array}{l}
                    \frac{c}{\sqrt{\nu / (\nu - 2)}} + \frac{c^2 \, \delta_a}{2 \nu / (\nu - 2)} + \frac{c^3 (\delta_a^2 - 1)}{6 \, \nu^{3/2} / (\nu - 2)^{3/2}} + \OO\big(\frac{1 + |\delta_a|^3}{\nu^4}\big)
                \end{array}
                \hspace{-1mm}\right\} \\
                &= \phi(\delta_a) \, \left\{\hspace{-1mm}
                \begin{array}{l}
                    \nu^{-1} \, d_1 + \nu^{-2} \, \big(\frac{\delta_a}{2} d_1^2 - d_1 + d_2\big) \\[1.5mm]
                    + \nu^{-3} \, \left(\hspace{-1mm}
                        \begin{array}{l}
                            \frac{1}{6} (\delta_a^2 - 1) d_1^3 - \delta_a d_1^2 - \frac{1}{2} d_1\\[0mm]
                            + \delta_a d_1 d_2 - d_2 + d_3
                        \end{array}
                        \hspace{-1mm}\right) + \OO\big(\frac{1 + |\delta_a|^3}{\nu^4}\big)
                \end{array}
                \hspace{-1mm}\right\}.
            \end{aligned}
        \end{equation}
        We also have the straightforward large deviation bounds
        \begin{equation}\label{eq:LD}
            \begin{aligned}
                &\int_{[a,\infty) \cap B_{\nu}^c(1/2)} \hspace{-0.6mm} f_{\nu}(x) \rd x = \OO(e^{-\beta \nu^{1/2}}), \\
                &\int_{[a,\infty) \cap B_{\nu}^c(1/2)} \hspace{-0.6mm} \phi(y) \rd y = \OO(e^{-\beta \nu^{1/2}}),
            \end{aligned}
        \end{equation}
        where $\beta > 0$ is a small enough constant, and the local approximation in Lemma~\ref{lem:LLT.Student} yields
        \begin{equation}\label{eq:Cressie.generalization.eq.4}
            \begin{aligned}
                \int_a^{\infty} f_{\nu}(x) \rd x - \int_{\delta_a}^{\infty} \hspace{-1mm} \phi(y) \rd y
                &= \nu^{-1} \bigg\{\frac{1}{4} \Psi_4(\delta_x) - \frac{3}{2} \Psi_2(\delta_x) + \frac{3}{4} \Psi(\delta_x)\bigg\} \\
                &\quad+ \nu^{-2} \left\{\hspace{-1mm}
                    \begin{array}{l}
                        \frac{1}{32} \Psi_8(\delta_x) - \frac{13}{24} \Psi_6(\delta_x) + \frac{41}{16} \Psi_4(\delta_x) \\[1.5mm]
                        - \frac{33}{8} \Psi_2(\delta_x) + \frac{41}{32} \Psi(\delta_x)
                    \end{array}
                    \hspace{-1mm}\right\} \\
                &\quad+ \nu^{-3} \left\{\hspace{-1mm}
                    \begin{array}{l}
                        \frac{1}{384} \Psi_{12}(\delta_x) - \frac{17}{192} \Psi_{10}(\delta_x) + \frac{127}{128} \Psi_8(\delta_x) \\[1.5mm]
                        - \frac{457}{96} \Psi_6(\delta_x) + \frac{1357}{128} \Psi_4(\delta_x) \\[1.5mm]
                        - \frac{651}{64} \Psi_2(\delta_x) + \frac{281}{128} \Psi(\delta_x)
                    \end{array}
                    \hspace{-1mm}\right\} \\
                &\quad+ \OO(\nu^{-4}),
            \end{aligned}
        \end{equation}
        where $\Psi_k(\delta_a) \leqdef \int_{\delta_a} y^k \phi(y) \rd y$.
        Now, using the fact that
        \begin{equation}
            \begin{aligned}
                &\Psi_{12}(\delta_a) = (10395 \delta_a + 3465 \delta_a^3 + 693 \delta_a^5 + 99 \delta_a^7 + 11 \delta_a^9 + \delta_a^{11}) \phi(\delta_a) + 10395 \Psi(\delta_a), \\
                &\Psi_{10}(\delta_a) = (945 \delta_a + 315 \delta_a^3 + 63 \delta_a^5 + 9 \delta_a^7 + \delta_a^9) \phi(\delta_a) + 945 \Psi(\delta_a), \\
                &\Psi_8(\delta_a) = (105 \delta_a + 35 \delta_a^3 + 7 \delta_a^5 + \delta_a^7) \phi(\delta_a) + 105 \Psi(\delta_a), \\
                &\Psi_6(\delta_a) = (15 \delta_a + 5 \delta_a^3 + \delta_a^5) \phi(\delta_a) + 15 \Psi(\delta_a), \\
                &\Psi_4(\delta_a) = (3 \delta_a + \delta_a^3) \phi(\delta_a) + 3 \Psi(\delta_a), \\
                &\Psi_2(\delta_a) = \delta_a \phi(\delta_a) + \Psi(\delta_a),
            \end{aligned}
        \end{equation}
        where $\Psi$ denotes the survival function of the standard normal distribution, Equations~\eqref{eq:c}, \eqref{eq:LD} and \eqref{eq:Cressie.generalization.eq.4} together yield
        \begin{equation}\label{eq:Cressie.generalization.eq.5}
            \begin{aligned}
                &\int_a^{\infty} \hspace{-0.6mm} f_{\nu}(x) \rd x - \int_{\delta_{a - c}}^{\infty} \hspace{-1mm} \phi(y) \rd y \\[-1.5mm]
                &\quad= \nu^{-1} \bigg\{\frac{1}{4}\delta_a^3 - \frac{3}{4} \delta_a - d_1\bigg\} \phi(\delta_a) + \nu^{-2} \left\{\hspace{-1mm}
                    \begin{array}{l}
                        \frac{1}{32} \delta_a^7 - \frac{31}{96} \delta_a^5 + \frac{91}{96} \delta_a^3 - \frac{41}{32} \delta_a \\[1.5mm]
                        - \big(\frac{\delta_a}{2} d_1^2 - d_1 + d_2\big)
                    \end{array}
                    \hspace{-1mm}\right\} \phi(\delta_a) \\[1mm]
                &\quad\quad+ \nu^{-3} \left\{\hspace{-1mm}
                    \begin{array}{l}
                        \frac{1}{384} \delta_a^{11} - \frac{23}{384} \delta_a^9 + \frac{29}{64} \delta_a^7 - \frac{305}{192} \delta_a^5 + \frac{1021}{384} \delta_a^3 - \frac{281}{128} \delta_a \\[2mm]
                        - \big(\frac{1}{6} (\delta_a^2 - 1) d_1^3 - \delta_a d_1^2 - \frac{1}{2} d_1 + \delta_a d_1 d_2 - d_2 + d_3\big)
                    \end{array}
                    \hspace{-1mm}\right\} \phi(\delta_a) + \OO(\nu^{-4}).
            \end{aligned}
        \end{equation}
        If we select $d_1 = d_2 = d_3 = 0$, then
        \begin{equation}\label{eq:Cressie.generalization.eq.5.order.0}
            \max_{a\in \R} \Big|\int_a^{\infty} f_{\nu}(x) \rd x - \int_{\delta_{a - c}}^{\infty} \hspace{-1mm} \phi(y) \rd y\Big| \leq \nu^{-1} \max_{a\in \R} \frac{|\delta_a|}{4} |\delta_a^2 - 3| \phi(\delta_a) + \OO(\nu^{-2}),
        \end{equation}
        which proves \eqref{eq:order.0.approx}.
        If we select $d_1 = \frac{\delta_a}{4} (\delta_a^2 - 3)$ and $d_2 = d_3 = 0$ to cancel the first brace in \eqref{eq:Cressie.generalization.eq.5}, then
        \begin{equation}\label{eq:Cressie.generalization.eq.5.order.1}
            \begin{aligned}
                &\max_{a\in \R} \Big|\int_a^{\infty} \hspace{-0.6mm} f_{\nu}(x) \rd x - \int_{\delta_{a - c}}^{\infty} \hspace{-1mm} \phi(y) \rd y\Big| \\[-0.5mm]
                &\leq \nu^{-2} \max_{a\in \R} \frac{|\delta_a|}{96} |13 \delta_a^4 - 88 \delta_a^2 + 195| \phi(\delta_a) + \OO(\nu^{-3}),
            \end{aligned}
        \end{equation}
        which proves \eqref{eq:order.1.approx}.
        If we select $d_1 = \frac{\delta_a}{4} (\delta_a^2 - 3)$, $d_2 = - \frac{\delta_a}{96} (13 \delta_a^4 - 88 \delta_a^2 + 195)$ and $d_3 = 0$ to cancel the first two braces in \eqref{eq:Cressie.generalization.eq.5}, then
        \begin{equation}\label{eq:Cressie.generalization.eq.5.order.2}
            \begin{aligned}
                &\max_{a\in \R} \Big|\int_a^{\infty} \hspace{-0.6mm} f_{\nu}(x) \rd x - \int_{\delta_{a - c}}^{\infty} \hspace{-1mm} \phi(y) \rd y\Big| \\[-0.5mm]
                &\leq \nu^{-3} \max_{a\in \R} \frac{|\delta_a|}{384} |35 \delta_a^6 - 293 \delta_a^4 + 1025 \delta_a^2 - 1767| \phi(\delta_a) + \OO(\nu^{-4}),
            \end{aligned}
        \end{equation}
        which proves \eqref{eq:order.2.approx}.
        If we select $d_1 = \frac{\delta_a}{4} (\delta_a^2 - 3)$, $d_2 = - \frac{\delta_a}{96} (13 \delta_a^4 - 88 \delta_a^2 + 195)$ and $d_3 = - \frac{\delta_a}{384} (35 \delta_a^6 - 293 \delta_a^4 + 1025 \delta_a^2 - 1767)$ to cancel the three braces in \eqref{eq:Cressie.generalization.eq.5}, then
        \begin{equation}\label{eq:Cressie.generalization.eq.5.order.3}
            \max_{a\in \R} \Big|\int_a^{\infty} \hspace{-0.6mm} f_{\nu}(x) \rd x - \int_{\delta_{a - c}}^{\infty} \hspace{-1mm} \phi(y) \rd y\Big| = \OO(\nu^{-4}),
        \end{equation}
        which proves \eqref{eq:order.3.approx}.
        This ends the proof.
    \end{proof}

    \begin{proof}[Proof of Theorem~\ref{thm:percentage.points}]
        Let $\alpha = S_{\nu}(\lambda)$ for some $\lambda\in B_{\nu}(\eta)$ and $\eta\in (0,1)$.
        By Theorem~\ref{thm:refined.approximations}, we have, as $\nu\to \infty$,
        \begin{equation}
            \left|\alpha - \Psi(\delta_{\lambda - \sum_{k=1}^i d_k / \nu^k})\right| \leq \frac{C_i'}{\nu^{i+1}}, \quad i\in \{0,1,2,3\},
        \end{equation}
        for some universal constants $C_i'$.
        A Taylor expansion for $\Psi$ at $\lambda$ yields, for $x = \lambda + \OO(\nu^{-1})$,
        \begin{equation}
            \frac{\Psi(x)}{\phi(\lambda)} = \frac{\Psi(\lambda)}{\phi(\lambda)} - \sum_{k=1}^i \frac{\phi^{(k-1)}(\lambda)}{\phi(\lambda) \, k!} (x - \lambda)^k + \OO(\nu^{-(i+1)}), \quad i\in \{0,1,2,3\},
        \end{equation}
        namely,
        \begin{equation}
            \begin{aligned}
                \frac{\Psi(x)}{\phi(\lambda)}
                &= \frac{\Psi(\lambda)}{\phi(\lambda)} - (x - \lambda) + \OO(\nu^{-2}), \quad \text{for } i = 1, \\[1mm]
                \frac{\Psi(x)}{\phi(\lambda)}
                &= \frac{\Psi(\lambda)}{\phi(\lambda)} - (x - \lambda) + \frac{\lambda}{2} (x - \lambda)^2 + \OO(\nu^{-3}), \quad \text{for } i = 2, \\[1mm]
                \frac{\Psi(x)}{\phi(\lambda)}
                &= \frac{\Psi(\lambda)}{\phi(\lambda)} - (x - \lambda) + \frac{\lambda}{2} (x - \lambda)^2 - \frac{(\lambda^2 - 1)}{6} (x - \lambda)^3 + \OO(\nu^{-4}), \quad \text{for } i = 3.
            \end{aligned}
        \end{equation}
        By applying these formulas with $x = \delta_{\lambda - \sum_{k=1}^i d_k / \nu^k}$, together with the fact that
        \begin{equation}
            \begin{aligned}
                \delta_{\lambda - \sum_{k=1}^i d_k / \nu^k} - \lambda
                &= \sqrt{\frac{\nu - 2}{\nu}} \left\{\lambda - \sum_{k=1}^i \frac{d_k}{\nu^k}\right\} - \lambda \\
                &= -\frac{\lambda}{\nu} - \frac{\lambda}{2 \nu^2} - \frac{\lambda}{2 \nu^3} + \OO(\nu^{-4}) - \left(1 - \frac{1}{\nu} - \frac{1}{2 \nu^2} + \OO(\nu^{-3})\right) \sum_{k=1}^i \frac{d_k}{\nu^k},
            \end{aligned}
        \end{equation}
        we get
        \begin{equation}
            \begin{aligned}
                \frac{\alpha}{\phi(\lambda)}
                &= \frac{\Psi(\lambda)}{\phi(\lambda)} + \left(\frac{\lambda}{\nu} + \frac{d_1}{\nu}\right) + \OO(\nu^{-2}), \quad \text{for } i = 1, \\[1mm]
                \frac{\alpha}{\phi(\lambda)}
                &= \frac{\Psi(\lambda)}{\phi(\lambda)} + \left(\frac{\lambda}{\nu} + \frac{\lambda}{2 \nu^2} + \frac{d_1}{\nu} + \frac{d_2}{\nu^2} - \frac{d_1}{\nu^2}\right) \\
                &\quad+ \frac{\lambda}{2} \left(\frac{\lambda^2}{\nu^2} + \frac{2 \lambda d_1}{\nu^2} + \frac{d_1^2}{\nu^2}\right) + \OO(\nu^{-3}), \quad \text{for } i = 2, \\[1mm]
                \frac{\alpha}{\phi(\lambda)}
                &= \frac{\Psi(\lambda)}{\phi(\lambda)} + \left(\frac{\lambda}{\nu} + \frac{\lambda}{2 \nu^2} + \frac{\lambda}{2 \nu^3} + \frac{d_1}{\nu} + \frac{d_2}{\nu^2} + \frac{d_3}{\nu^3} - \frac{d_1}{\nu^2} - \frac{d_2}{\nu^3} - \frac{d_1}{2 \nu^3}\right) \\
                &\quad+ \frac{\lambda}{2} \left(\frac{\lambda^2}{\nu^2} + \frac{2 \lambda d_1}{\nu^2} + \frac{d_1^2}{\nu^2} + \frac{\lambda^2}{\nu^3} + \frac{2 \lambda d_2}{\nu^3} - \frac{\lambda d^1}{\nu^3} + \frac{2 d_1 d_2}{\nu^3} - \frac{2 d_1^2}{\nu^3}\right) \\
                &\quad+ \frac{(\lambda^2 - 1)}{6} \left(\frac{\lambda^3}{\nu^3} + \frac{d_1^3}{\nu^3}\right) + \OO(\nu^{-4}), \quad \text{for } i = 3.
            \end{aligned}
        \end{equation}
        This ends the proof.
    \end{proof}

%\section*{Data availability statement}

%No data were used to support this study.

\section*{Funding}

F.\ Ouimet is supported by postdoctoral fellowships from the NSERC (PDF) and the FRQNT (B3X supplement and B3XR).

\section*{Conflicts of interest}

The author declares no conflict of interest.

%
% ----------  B I B L I O G R A P H Y  ----------
%

\section*{References}
\phantomsection
\addcontentsline{toc}{chapter}{References}

\bibliographystyle{authordate1}
\bibliography{Ouimet_2022_LLT_Student_bib}

\end{document}